\documentclass[final]{siamart171218}
\usepackage{amsmath,amssymb,color,soul,enumerate,mathrsfs,algpseudocode,subcaption}
\usepackage{tabularx}

\allowdisplaybreaks

\newcommand{\Q}{\mathbb Q}
\newcommand{\N}{\mathbb{N}}
\newcommand{\Z}{\mathbb{Z}}
\newcommand{\R}{\mathbb{R}}
\newcommand{\conv}{\operatorname{conv}}

\newcommand{\intr}{\operatorname{int}}

\newcommand{\cl}{\operatorname{cl}}

\newcommand{\floor}[1]{\left\lfloor#1\right\rfloor}

\renewcommand{\epsilon}{\varepsilon}

\def\ve#1{\mathchoice{\mbox{\boldmath$\displaystyle\bf#1$}}
{\mbox{\boldmath$\textstyle\bf#1$}}
{\mbox{\boldmath$\scriptstyle\bf#1$}}
{\mbox{\boldmath$\scriptscriptstyle\bf#1$}}}

\newcommand{\x}{{\ve x}}

\newcommand{\z}{{\ve z}}

\renewcommand{\b}{{\ve b}}

\newcommand{\lp}{\textsf{LP}}
\newcommand{\gmi}{\textsf{GMI}}
\newcommand{\ip}{\textsf{IP}}
\newcommand{\X}{\textsf{X}}
\newcommand{\GX}{\textsf{GX}}
\newcommand{\XG}{\textsf{XG}}
\newcommand{\GXG}{\textsf{GXG}}
\newcommand{\Best}{\textsf{Best}}

\newsiamthm{conjec}{Conjecture}
\newsiamthm{cor}{Corollary}
\newsiamthm{prop}{Proposition}
\newsiamthm{claim}{Claim}
\newsiamremark{question}{Question}
\newsiamthm{conj}{Conjecture}
\newsiamthm{obs}{Observation}
\newsiamthm{example}{Example}

\newsiamthm{REMARK}{Remark}

\newcommand{\aff}{\operatorname{aff}}

\numberwithin{equation}{section}

\DeclareMathOperator* {\Flt}{Flt}

\title{Can cut generating functions be good {\em and} efficient?\thanks{Submitted to the editors.
\funding{Both the authors gratefully acknowledge the support from NSF grant CMMI1452820}}}

\author{Amitabh Basu\thanks{Department of Applied Mathematics and Statistics, Johns Hopkins University, Baltimore, MD (\email{abasu9@jhu.edu}).}
\and Sriram Sankaranarayanan \thanks{Department of Civil Engineering, Johns Hopkins University, Baltimore, MD (\email{srirams@jhu.edu}).}}

\begin{document}
\maketitle

\begin{abstract} Making cut generating functions (CGFs) computationally viable is a central question in modern integer programming research. One would like to find CGFs that are simultaneously {\em good}, i.e., there are good guarantees for the cutting planes they generate, and {\em efficient}, meaning that the values of the CGFs can be computed cheaply (with procedures that have some hope of being implemented in current solvers). We investigate in this paper to what extent this balance can be struck. We propose a family of CGFs which, in a sense, achieves this harmony between {\em good} and {\em efficient}. In particular, {we provide a parameterized family of $b+\Z^n$ free sets to derive CGFs from  and} show that our proposed CGFs give a good approximation of the closure given by CGFs obtained from {all} maximal $b+\Z^n$ free sets and their so-called {\em trivial liftings}, and simultaneously, show that these CGFs can be computed with explicit, efficient procedures. {We provide a constructive framework to identify these sets as well as computing their trivial lifting. We follow it up with computational experiments to demonstrate this and to evaluate their practical use. Our proposed family of cuts seem to give some tangible improvement on randomly generated instances compared to GMI cuts; however, in MIPLIB 3.0 instances, and vertex cover and stable problems on random graph instances, their performance is poor.} 
\end{abstract}



\section{Introduction}
In this paper, we study the inequality description of sets of the form
\begin{align}\label{eq:cor-poly}
X(R,P) \quad &:= \quad \conv\left\{ (s,y)\in \mathbb{R}^k_+\times \mathbb{Z}^\ell_+\mid Rs+Py \in b+\mathbb{Z}^n \right\}
\end{align}
where $n, k, \ell \in \N$, $R\in \mathbb{R}^{n\times k}, P\in\mathbb{R}^{n\times \ell}, b\in\mathbb{R}^n\setminus \mathbb{Z}^n$. Such sets have been the focus of intense study in the last decade, and are typically refereed to as {\em mixed-integer corner polyhedra} in the literature; see the surveys~\cite{delpia-4OR,basu2015geometric,basu2016light,basu2016light2} and~\cite[Chapter 6]{conforti2014integer}, and the references therein. One of the focal points in this recent activity has been the revival of the {\em cut generating function} approach, originally pioneered by Gomory and Johnson in their seminal work in the 1970s~\cite{infinite,infinite2,johnson}. The phrase ``cut generating function" was invented relatively recently by the authors of~\cite{conforti2014cut}.


\begin{definition}[Valid pair]\label{Def:ValidPair}
Fix $n\in \N$. A pair of real valued functions $\left ( \psi,\pi\right )$ on $\R^n$ are said to be a \emph{valid pair} if  
\begin{align}\label{eq:cut}
    \sum_{i=1}^k\psi(r_i)s_i + \sum_{i=1}^\ell \pi(p_i)y_i \quad&\geq\quad 1 
\end{align}
 is a valid inequality for $X(R,P)$ for all $k, \ell, R, P$, where $r_i$ and $p_i$ refer to the columns of $R$ and $P$ respectively. 
\end{definition}

The important thing to note is that a valid pair of functions only depends on the dimension $n$ and $b$, and should work for any matrices $R,P$ with $n$ rows, and an arbitrary number of columns. Gomory and Johnson made the discovery that not only do such valid pairs of functions exist, they give a unifying framework for many cut generating procedures extensively used in the integer programming community. {Gomory's original motivation~\cite{gom} was to choose $n$ rows from the optimal simplex tableaux of a general mixed-integer optimization problem and apply these cut generating functions (for this particular choice of $n$ rows of the tableaux) to obtain cutting planes for the original problem.} The modern trend has been to build a more computationally tractable viewpoint of this theory. This has been possible by drawing upon novel insights into cutting plane theory by Balas from the 1970s, which was termed by him as the theory of {\em intersection cuts}~\cite{bal}. We summarize this approach to cut generating functions next.

 Given a convex set $C$ with the origin in its interior, the {\em gauge function} is defined as
       $ \psi_C(x) :=\inf_{\lambda > 0 }\left\{ \lambda: \frac{x}{\lambda} \in C \right\}$.
	   Let $S$ be any closed subset of $\R^n\setminus\{0\}$ (not necessarily convex). A closed convex set $B$ containing $0$ in its interior is said to be an {\em $S$-free convex neighborhood of $0$} if $\intr(B) \cap S = \emptyset$. It is said to be a {\em maximal} $S$-free convex neighborhood of $0$ if it is not strictly contained in another $S$-free convex neighborhood of $0$. For brevity, we will often refer to such sets as (maximal) $S$-free convex sets. In this paper, we will be concerned with $S=b+ \Z^n$, where $b \in \R^n \setminus \Z^n$. The starting point of combining Balas' intersection cuts and Gomory-Johnson's cut generating function theory is the observation that setting $\psi = \pi = \psi_B$, where $B\subseteq \R^n$ is a maximal $b+\Z^n$ free set gives a valid pair. Thus, for every maximal $b+\Z^n$ free set $B\subseteq \R^n$, we obtain a valid inequality $ \sum_{i=1}^k\psi_B(r_i)s_i + \sum_{i=1}^\ell \psi_B(p_i)y_i \geq1 $ for $X(R,P)$, for all $k, \ell, R,P$. Such inequalities can be implemented in a cut generating procedure in any modern solver, as long as one has a way of computing $\psi_B(r)$ efficiently, for any $r\in \R^n$. Here, a new ingredient has been added by modern research, which uses a result of Lovasz~\cite{lovasz} (later refined by others) stating that all maximal $b+\Z^n$ free sets are polyhedra that can be written in the form $B:= \{x \in \R^n: a^i\cdot x \leq 1, \;\; i = 1, \ldots, m\}$, where $a^i \in \R^n$. It turns out that the gauge function of such a set is simply $\psi_B(r) = \max_{i=1}^m a^i\cdot r$. This now makes the computation of the coefficients of the cut $ \sum_{i=1}^k\psi_B(r_i)s_i + \sum_{i=1}^\ell \psi_B(p_i)y_i \geq1 $ more concrete, compared to the original theory of Gomory and Johnson.

The next ingredient in the modern approach to cut generating functions is to use an idea due to Balas and Jeroslow~\cite{baljer}, which they termed {\em monoidal strengthening}. In our context, the observation translates to the fact that one can improve the coefficients of the $y_i$ variables, by using the integrality constraint on these variables.

\begin{definition}[Trivial lifting]
    Let $b \in \R^n\setminus \Z^n$ and let $B$ be a maximal $(b+\Z^n)$-free convex set. The \emph{trivial lifting} of $\psi_B(x)$ is defined by 
\begin{align}\label{eq:trivial-lifting}
    \widetilde{\psi_B}(x) \quad&=\quad \min \left ( 1, \inf_{z\in \mathbb{Z}^n}\psi_B(x+z)\right )
\end{align}
\end{definition}

One of the main outcomes of the recent computational perspective on cut generating functions can be summarized as follows~\cite{dw,infinite}.

\begin{theorem}\label{thm:comp-CGF} Let $b \in \R^n\setminus \Z^n$ and let $B$ be a maximal $(b+\Z^n)$-free convex set. 
Then $(\psi_B, \widetilde{\psi_B})$ is a valid pair.
\end{theorem}

It is important to note that given a maximal $b+\Z^n$ free set $B$, there may exist several functions $\pi: \R^n \to \R$ such that $(\psi_B, \pi)$ is a valid pair; all such functions $\pi$ are called {\em liftings} of $\psi_B$. The trivial lifting is only one such function. Since the variables $y$ are nonnegative, if we have two liftings $\pi_1 \leq \pi_2$, then the cutting plane~\eqref{eq:cut} derived from $\pi_2$ is dominated by the one derived from $\pi_1$. Thus, ideally, one would like to work with {\em minimal liftings}, i.e., liftings $\pi$ such that there does not exist a different lifting $\pi'\neq \pi$ with $\pi'\leq \pi$.  In general, the trivial lifting may not be minimal; characterizing situations when it is indeed minimal has received a lot of attention~\cite{averkov2015lifting,basu2012unique,basu-paat-lifting,dw,dw2,bcccz,ccz}. {In fact, the trivial lifting is always an upper bound on any minimal lifting, i.e., $\pi \leq \widetilde\psi$ for any minimal lifting $\pi$ of $\psi$. Thus, when the trivial lifting is minimal, it is the unique minimal lifting.}
\bigskip

In our opinion, there are two key obstacles to implementing such cut generating functions in state-of-the-art software: \begin{enumerate}  \item There are too many (in fact, infinitely many) maximal $b+\Z^n$ free sets to choose from. This is the problem of {\em cut selection}.\item For maximal $b+\Z^n$ free polyhedra with complicated combinatorial structure, the computation of the trivial lifting via~\eqref{eq:trivial-lifting} is extremely challenging. Moreover, computing the values of  minimal liftings, especially if the trivial lifting is not the unique minimal lifting is even more elusive, with no formulas like~\eqref{eq:trivial-lifting} available.\end{enumerate} Thus, a central question in making cut generating function theory computationally viable, which also motivates the title of this paper, is the following. 

\begin{quote} \begin{question}\label{quest:find-subset} Find a ``simple" subset of maximal $b+\Z^n$ free polyhedra such that two goals are simultaneously achieved: 
\begin{itemize} \item[(i)] provide guarantees that this ``simple'' subset of $b+\Z^n$ free sets gives a good approximation of the closure obtained by throwing in cuts from all possible maximal $b+\Z^n$ free sets, and \item[(ii)] cutting planes like~\eqref{eq:cut} can be derived from them with relatively light computational overhead, either via trivial liftings or other lifting techniques. \end{itemize} \end{question} 
\end{quote}


\subsection{Summary of results} The goal of this paper is to make some progress in \cref{quest:find-subset}. 
{In our opinion, these results provide both theoretical evidence for the utility of cut generating functions and algorithms that are efficient enough be implemented in practice.} 
\begin{enumerate}
    \item\label{result-1} One may wonder if the trivial lifting function of the gauge can approximate any minimal lifting up to some factor. We show that there exist maximal $b+\Z^n$ free sets whose gauge functions have minimal liftings that are arbitrarily better than the trivial lifting (on some subset of vectors) {[recall that any minimal lifting is pointwise smaller than the trivial lifting]}. 
 More formally, we establish 
    \begin{theorem}\label{thm:bad-approx}
    Let $n$ be any natural number and $\epsilon > 0$. There exists $b \in \R^n\setminus \Z^n$ and a family $\mathcal{F}$ of maximal $(b+\Z^n)$-free sets such that for any $B\in \mathcal{F}$, there exists a minimal lifting $\pi$ of $\psi_B$ and $p \in \R^n$  satisfying $\frac{\pi(p)}{\widetilde{\psi_B}(p)} < \epsilon$.
    \end{theorem}

\item\label{result-2} 
Given an arbitrary maximal $b+\Z^n$ free set $B$, computing the trivial lifting using~\eqref{eq:trivial-lifting} can be computationally very hard because it is equivalent to the notorious closest lattice vector problem in the algorithmic geometry of numbers literature~\cite{Eisenbrand2010}. One could potentially write an integer linear program to solve it, but this somewhat defeats the purpose of cut generating functions: one would like to compute the coefficients much faster than solving complicated optimization problems like~\eqref{eq:trivial-lifting} (and even harder IPs for general lifting). To overcome this issue, we isolate a particular family of maximal $b+\Z^n$ free sets that we call {\em generalized cross-polyhedra} (see \cref{Def:gen-cross-poly} for a precise definition) and give an algorithm for computing the trivial lifting function for any member of this family without using a high dimensional integer linear program. For this family, one needs $O(2^n)$ time to compute the gauge function because the $b+\Z^n$ free sets have $2^n$ facets, and one needs an additional $O(n2^n)$ time to compute the trivial lifting coefficient. {\em Recall that $n$ corresponds to the number of rows used to generate the cuts}. This is much better complexity compared to solving~\eqref{eq:trivial-lifting} using an integer program or a closest lattice vector (the latter will have to deal with an asymmetric, polyhedral gauge which is challenging). This is described in \cref{sec:Algorithm_for_trivial_liftings}; see \cref{alg:CoProdLift}. For a subfamily of generalized cross-polyhedra, both of these computations (gauge values and trivial lifting values) can actually be done in $O(n)$ time, which we exploit in our computational tests (see \cref{sec:cut-generation}). We envision using this in software and computations in the regime $n \leq 15$. {\em To the best of our knowledge, no previous work provides a comparable lifting procedure that can be easily coded in software and that works for any number of rows $n$, even for a restricted class of $b+\Z^n$ free sets. Previous work on lifting that can be readily translated to code, without solving an intermediate IP, has focused on the $n=1, 2$ case (see the relevant literature discussed below).}
     
    \item\label{result-3} From a theoretical perspective, we also show that our family of generalized cross-polyhedra can provide a finite approximation for the closure of cutting planes of the form 
        $$\sum_{i=1}^k\psi_B(r_i)s_i + \sum_{i=1}^\ell \widetilde{\psi_B}(p_i)y_i \geq 1.$$ More precisely, for any matrices $R \in \R^{n \times k}, P \in \R^{n\times \ell}$, and any maximal $b+\Z^n$ free set $B$, let $H_{B}(R,P) := \{(s,y) : \sum_{i=1}^k\psi_B(r_i)s_i + \sum_{i=1}^\ell \widetilde{\psi_B}(p_i)y_i \ge 1\}$. Let $\mathcal{G}_b$ denote the set of all generalized cross-polyhedra (as applicable to $S = b + \Z^n$). 
       Then, we have 
    
    \begin{theorem}\label{thm:approx}
   Let $n\in \N$ and $b \in \Q^n\setminus \Z^n$. Define for any matrices $R, P$
   \begin{align*}
   M(R,P) &:= \cap_{B \textrm{ maximal $b+\Z^n$ free set}} H_B(R,P)\\
   G(R,P) &:= \cap_{B \in \mathcal{G}_b} H_B(R,P)	
   \end{align*}   
   Then there exists a constant $\alpha$ depending only on $n, b$ such that $M(R,P) \subseteq G(R,P) \subseteq \alpha M(R,P)$ for all matrices $R,P$.
   \end{theorem}
   
   Note that since $\psi_B, \widetilde\psi_B \geq 0$, both $M(R,P)$ and $G(R,P)$ in \cref{thm:approx} are polyhedra of the blocking type, i.e., they are contained in the nonnegative orthant and have their recession cone is the nonnegative orthant. Thus, the relationship $G(R,P) \subseteq \alpha M(R,P)$ shows that one can ``blow up" the closure $M(R,P)$ by a factor of $\alpha$ and contain $G(R,P)$. Equivalently, if we optimize any linear function over $G(R,P)$, the value will be an $\alpha$ approximation compared to optimizing the same linear function over $M(R,P)$. 

   \item\label{result-4} We test our family of cutting planes on randomly generated mixed-integer linear programs, on vertex cover and stable set problems in random graphs, and on the MIPLIB 3.0 set of problems. The short summary is that we seem to observe a tangible improvement with our cuts on the general random instances, {\em no improvement whatsoever} in the random graph instances, and no significant improvement on structured problems like MIPLIB 3.0 problems (except for a specific family). The random data set consists of approx. 13000 instances, and our observed improvement cannot be explained by random noise. More details are available in \cref{sec:CompExpr}.
   
{Our conclusion is that while the family of generalized cross polyhedra has a closure with good properties (like Theorem~\ref{thm:approx} above) and any particular cut from the family can be generated with light computational overhead (point 2. above), the {\em cut selection problem} is overwhelming even for this specialized family. We used a very naive random sampling method for selecting cuts from this family and clearly this heuristic is not good enough, as our computational results show. Our efforts at approximating the closure did not report anything different (see discussion in Section~\ref{sec:approximate-closure}). }

{The one encouraging message we draw from our computational experience is that in the general random instances distinct gain was observed in a non-trivial fraction (about 10\%). Perhaps this suggests that the cuts are able to exploit some structure in dense MIP problems. But what this structure could be is not very clear.}
   \end{enumerate}%
   

\subsection{Discussion}    {We isolate a {\em parametrizable} family of $b+ \Z^n$ free sets such that the cut generating functions derived from them are simultaneously ``good" in the sense that their closure provides a good approximation to the closure of cuts obtained from all $b+\Z^n$ free sets, and ``efficient" in the sense that the cut coefficients can be computed in a few lines of computer code. We are unaware of a similar result on cut generating functions from the literature (we do a more detailed literature review below). }

   {While there are results in prior literature (discussed in the next subsection) that show the existence of ``good'' families in the sense of approximations, one potential concern with these families is the following. It seems impossible to give a ``nice" parametrization of these families from [3] that can be exploited computationally. In contrast, the family we propose in this manuscript can be parametrized very cleanly by tuples of the form $(\gamma,\mu,U)$ where $\gamma \in \R^n, \mu\in \Delta^{n-1}$ ($\Delta^{n-1}$ is the standard simplex in $\R^n$) and $U \in \R^{n\times n}$ is a unimodular matrix ($n$ refers to the chosen number of rows from the simplex tableaux on which the analysis is being done). }

   {Moreover, the problem of actually computing the cut coefficients is highly non-trivial for these ``good" families from the literature (involving closest lattice vector problems, as discussed in point 2. above). The only family of sets in previous literature where more efficient algorithms exist to compute \emph{any} lifting is the family of 2-dimensional $b+\Z^n$ free convex sets and even there, it is ironically quite non-trivial to compute the trivial lifting~\cite{fukasawa2016not}. But for the ``good" family we propose above, even in arbitrary dimensions, we give an efficient algorithm to compute the trivial lifting (which also happens to be the unique minimal lifting).}

{We view the computation section as a proof-of-concept to illustrate that each step mentioned in the paper --- constructing the generalized crosspolyhedra, computing their gauge and computing the trivial trival lifting --- is constructive and hence implementable. That said, we have not been able to address the {\em cut selection problem} adequately in practice. Our family is still ``too big" in spite of being ``efficient" in the sense described above, and our heuristics for selecting cutting planes from this family were unable to provide the theoretical gains promised by the closure. We view the results in this paper as making some partial progress towards answering Question~\ref{quest:find-subset}. There is no doubt that more advances are needed towards settling this question in a completely satisfying manner.}

\subsection{Related literature and discussion} It would be hard to list the numerous papers that have appeared in the last decade pertaining to cut generating functions. We refer to the reader to the recent surveys \cite{delpia-4OR,basu2015geometric,basu2016light,basu2016light2} and~\cite[Chapter 6]{conforti2014integer}, and the references therein. There are some papers worth singling out as they relate more directly to the flavor of questions we investigate in this paper. 

In~\cite{fukasawa2016intersection,fukasawa2016not}, the authors are explicitly concerned with computing the trivial lifting formula~\eqref{eq:trivial-lifting}, without solving an integer linear program. In fact, our result outlined in \cref{result-2} above is very much inspired by ideas from~\cite{fukasawa2016not}. This, to the best of our knowledge, summarizes the most directly comparable literature on the {\em efficiency} aspect of cut generating functions. There also has been parallel work on the {\em goodness} aspect. The papers~\cite{bbcm,DBLP:journals/siamjo/AndersenWW09,basu2012intersection,campelo2009stable,cornuejols2012tight,he2011probabilistic,del2011probabilistic,basu2011probabilistic,del2012rank,averkov2017approximation} provide results that, from a rigorous mathematical perspective, either show that a certain subset of cut generating functions forms a good approximation, or some natural subset (like split cuts) forms a bad approximation in the worst case. 

In general, testing of cut generating functions computationally, with and without the trivial lifting, has been done in~\cite{lp,poirrier2012multi,louveaux2015strength,basu2011experiments,dey2010experiments,espinoza2010computing}. Perhaps the best summary of these investigations is a quote from Conforti, Cornu\'ejols and Zambelli~\cite{corner_survey}: ``Overall, the jury is still out on the practical usefulness of [cut generating functions]" (the part in brackets is our paraphrasing of the original quote). Nevertheless, it is our firm belief that this only indicates further investigations with a computational perspective are needed in this area. We hope the results of this paper can guide this research. While our computational experience adds to the ambiguity of whether these new cutting plans are useful in practice, it is heartening (at least to us) to see the appreciable advantage observed in random instances. Moreover, some of the positive results reported in~\cite{espinoza2010computing} came from using special cases of our construction of generalized cross-polyhedra.

\subsection{Outline} The remainder of the paper is dedicated to rigorously establishing the above results. \Cref{sec:gcp-approx} formally introduces the class of generalized cross-polyhedra and \cref{thm:approx} is proved. \Cref{sec:Algorithm_for_trivial_liftings} then gives an algorithm for computing the trivial lifting for the family of generalized cross-polyhedra, which avoid solving integer linear programming problems or closest lattice vector problems for this purpose. \Cref{sec:CompExpr} gives the details of our computational testing. \Cref{sec:bad-approx} proves \cref{thm:bad-approx}.

\section{Approximation by Generalized Cross Polyhedra}\label{sec:gcp-approx}


\begin{definition}\label{Def:gen-cross-poly}[Generalized cross-polyhedra] We define the family of {\bf generalized cross-polytopes} recursively. For $n=1$, a generalized cross-polytope is simply any interval $I_a:= [a,a+1]$, where $a \in \Z$. For $n \geq 2$, we consider any generalized cross-polytope $B \subseteq \R^{n-1}$, a point $c \in B$, $\gamma \in \R$, and $\mu \in (0,1)$. A generalized cross-polytope in $\R^n$ built out of $B, c, \gamma, \mu$ is defined as the convex hull of $\left (\frac{1}{\mu}(B-c) + c\right ) \times \{\gamma\}$ and $\{c\} \times \left (\frac{1}{1-\mu}(I_{\lfloor \gamma\rfloor} - \gamma) + \gamma\right )$. The point $(c, \gamma) \in \R^n$ is called the {\bf center} of the generalized cross-polytope.

\begin{figure}[t]
	\captionsetup{justification=centering}
    \captionsetup[subfigure]{justification=centering}
	\centering
	\begin{subfigure}[b]{0.45\textwidth}
		\includegraphics[width=\textwidth]{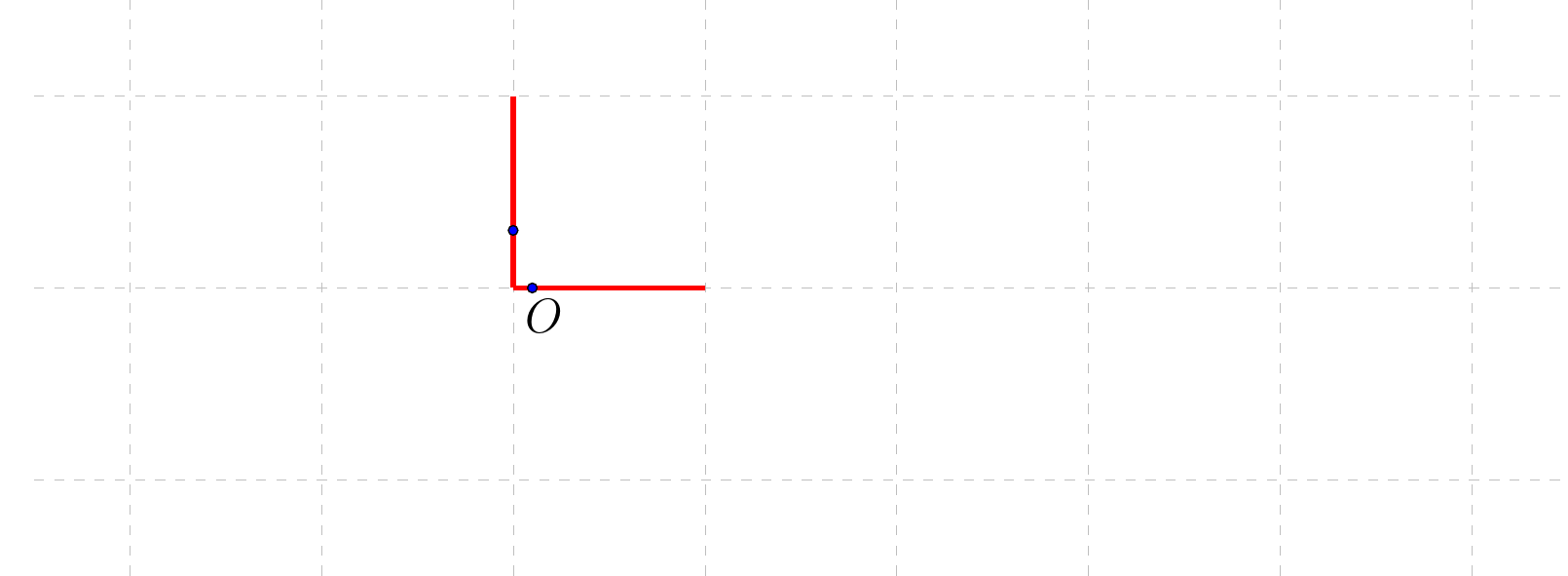}
		\caption{The horizontal red line is the crosspolytope $B$ and the vertical red line represents the interval $I_{\floor{\gamma}}$. The points on $B$ and $I_{\floor{\gamma}}$ are $c$ and $\gamma$ respectively.}\label{fig:CrossPoly1}
	\end{subfigure}\hfill
	\begin{subfigure}[b]{0.45\textwidth}
		\includegraphics[width=\textwidth]{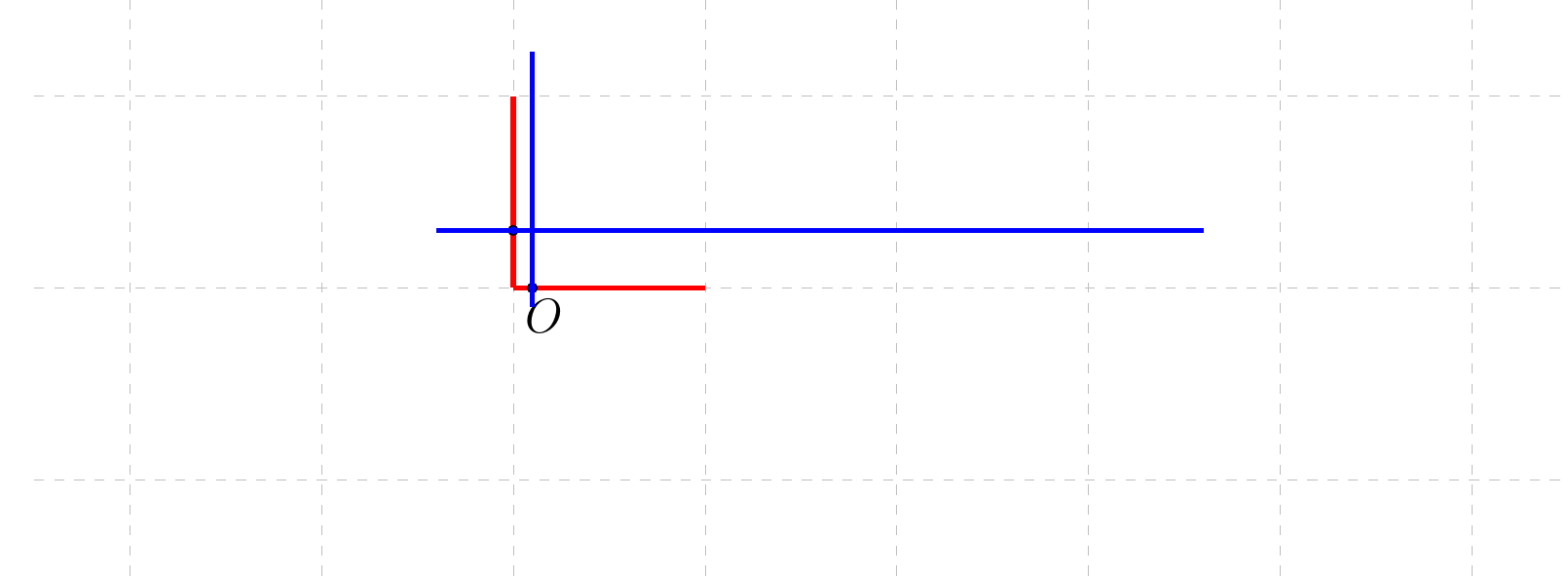}
		\caption{With $\mu = 0.25$, the horizontal blue line is $\left( \frac{1}{\mu}(B-c)+c \right) \times \left\{ \gamma \right\}$ and the vertical line is $\left\{ c \right\} \times \left( \frac{1}{1-\mu}\left( I_{\floor{\gamma}}-\gamma \right)+\gamma \right)$.}\label{fig:CrossPoly2}
	\end{subfigure}
	\begin{subfigure}[b]{0.45\textwidth}
		\includegraphics[width=\textwidth]{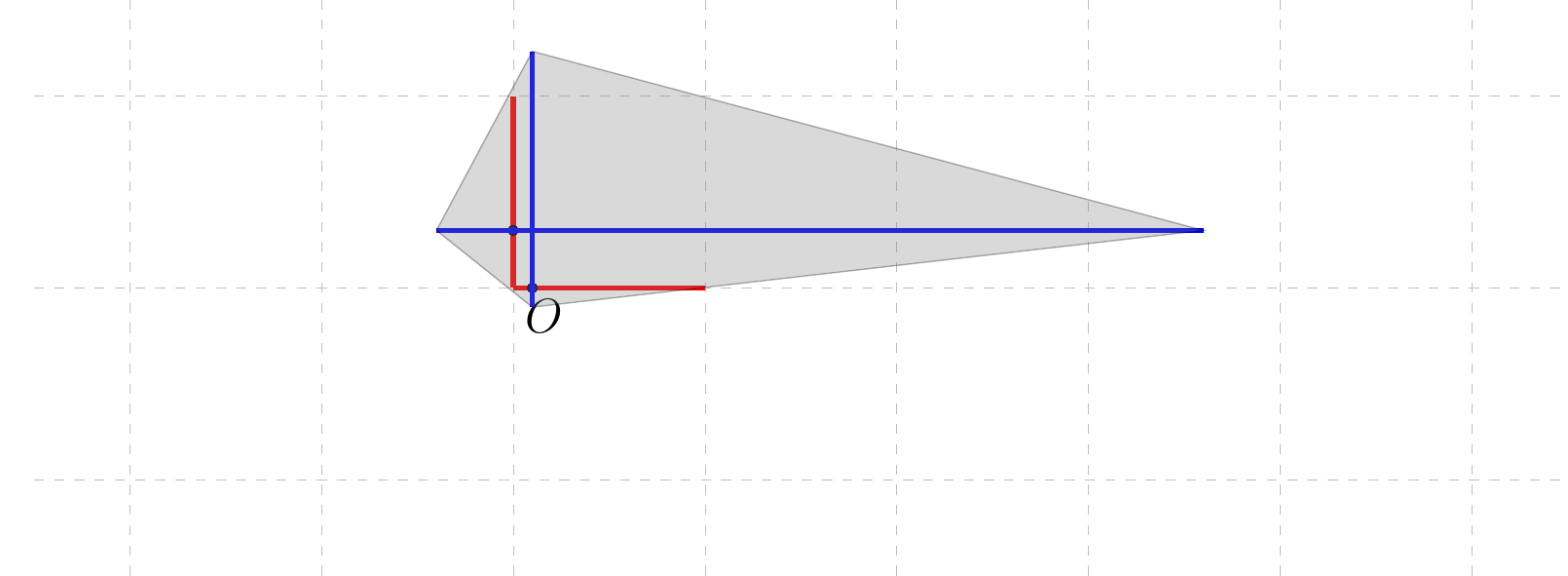}
		\caption{The convex hull of the sets in \cref{fig:CrossPoly2} gives $G$, the generalized cross-polytope.}\label{fig:CrossPoly3}
	\end{subfigure}\hfill
	\begin{subfigure}[b]{0.45\textwidth}
		\includegraphics[width=\textwidth]{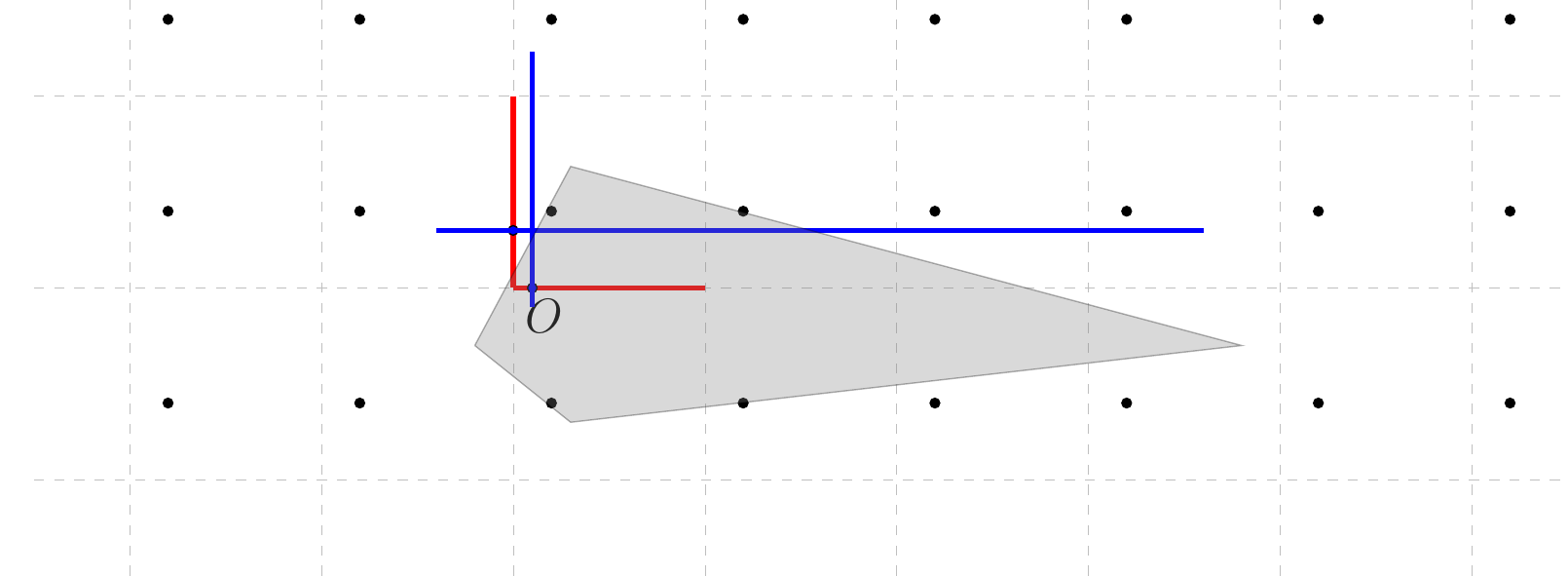}
		\caption{$b+G$ is the new $b+\Z^n$ free generalized cross-polytope.}\label{fig:CrossPoly4}
	\end{subfigure}
	\caption{Cross-polytope construction - The points $b+\Z^n$ are shown as black dots and the points in $\Z^n$ are the intersection of the dotted grid.}\label{fig:CrossPoly}
\end{figure}

A {\bf generalized cross-polyhedron} is any set of the form $X \times \R^{n-m}$, where {$m<n$} and  $X \subseteq \R^{m}$ is a generalized cross-polytope in $\R^m$. 
\end{definition}

The following theorem collects important facts about generalized cross-polyhedra that were established in~\cite{averkov2015lifting,basu-paat-lifting} (where these sets were first defined and studied) and will be important for us below. 

\begin{theorem}\label{thm:gcp-facts} Let $G\subseteq \R^n$ be a generalized cross-polyhedron. The following are all true.
\begin{itemize}
	\item[(i)] Let $b \in \R^n\setminus \Z^n$ such that $-b \in \intr(G)$. 
	Then $b+G$ is a maximal $b+\Z^n$ free convex set.	Moreover, using the values of $c, \gamma$ and $\mu$ in the recursive construction, one can find normal vectors $a^1, \ldots, a^{2^n} \in \R^n$ such that $b + G = \{x \in \R^n : a^i\cdot x \leq 1, \;\; i =1, \ldots, 2^n\}$.
\item[(ii)] If $G$ is a generalized cross-polytope, then there exists a unique $z\in \Z^n$ such that $z+[0,1]^n \subseteq G \subseteq \cup_{j=1}^n ((z+[0,1]^n) + \ell_j )$, {where $\ell_j$ is the line in $\R^n$ through the origin in the direction of the $j$-th unit vector.} Moreover, $z_j = \lfloor \gamma_j \rfloor$, where $\gamma_j$ is the value used in the $j$-th stage in the recursive construction of $G$ for $j=1, \ldots, n$ (for $j=1$, $\gamma_1$ is taken to be the left end point of the interval used to start the construction).
\end{itemize}
\end{theorem}

{Part (i) of Theorem~\ref{thm:gcp-facts} follows from~\cite[Theorem 5.3]{averkov2015lifting}, or its generalization~\cite[Theorem 4.1]{basu-paat-lifting}. Part (ii) follows from a straightforward inductive argument which we omit in this paper.}

Next we recall the definition of lattice width and the flatness theorem, which we need to prove \cref{thm:approx}.
\begin{definition}[Width function and lattice width]\label{Def:width}
    For every nonempty subset $X\subset \mathbb{R}^n$, the \emph{width function } $w(X,\circ):\mathbb{R}^n\mapsto[0,\infty]$ of $X$ is defined to be 
    \begin{align}
        w(X,u) \quad&:=\quad \sup_{x\in X}x\cdot u - \inf_{x\in X}x.u
    \end{align}
    The \emph{lattice width} of $X$ is defined as 
    \begin{align}
        w(X) \quad&:=\quad \inf_{u\in \mathbb{Z}^n\setminus \left \lbrace 0\right \rbrace}w(X,u)
    \end{align}
\end{definition}
\begin{definition}[Flatness]\label{Def:Flt}
    The \emph{Flatness function}  is defined as
    \begin{align}
        \Flt(n) \quad&:=\quad \sup \left \lbrace w(B): B \text{ is a $b+\Z^n$ free set in } \mathbb{R}^n\right \rbrace
    \end{align}
\end{definition}
\begin{theorem}\cite[Flatness theorem]{Barvinok2002}\label{thm:Flt}
    $\Flt(n)\leq n^{{5}/{2}}$ for all $n \in \N$.
\end{theorem} 

The main goal of this section is to establish the following result, which immediately implies \cref{thm:approx}.


\begin{theorem}\label{thm:gcp-main} Let $b \in \Q^n\setminus \Z^n$ such that the largest denominator in a coordinate of $b$ is $s$. Let $L$ be a $b+\Z^n$ free set with $0\in \intr(L)$.  
Then there exists a generalized cross-polyhedron $G$ such that $B:=b + G$ is a $b+\Z^n$ free convex set such that $\left(\frac{1}{s4^{n-1}\Flt(n)}\right)^{n-1} L \subseteq B$.
\end{theorem}

Let us quickly sketch why \cref{thm:gcp-main} implies \cref{thm:approx}. 

\begin{proof}[Proof of \cref{thm:approx}] We claim that  $\alpha = \left (s4^{n-1}\Flt(n)\right)^{n-1}$ works. {Gauge functions satisfy the properties that $A \subseteq B$ implies that $\psi_A \geq \psi_B$, and $\psi_{\gamma A} = \frac{1}{\gamma}\psi_A$ for any $\gamma\geq 0$~\cite{rock}. Thus,} \cref{thm:gcp-main} implies that for any maximal $b+\Z^n$ free set $L$, there exists a generalized cross-polyhedron $B$ such that $\psi_B\leq \alpha \psi_L$, consequently, by~\eqref{eq:trivial-lifting}, $\tilde\psi_B\leq \alpha \tilde\psi_L$. Thus, $H_{B}(R,P)\subseteq \alpha H_L(R,P)$ and we are done.
\end{proof}

The rest of this section is dedicated to proving \cref{thm:gcp-main}. We need to first introduce some concepts and intermediate results, and the final proof of \cref{thm:gcp-main} is assembled at the very end of the section.


\begin{definition}[Truncated cones and pyramids]\label{Def:Pyramid}
Given an $n-1$-dimensional closed convex set $M\subset \mathbb{R}^{n}$, a vector $v\in \mathbb{R}^n$ such that $\aff(v+M) \neq \aff(M)$, and a scalar $\gamma\in \R_+$, we say that the set $T(M,v,\gamma):=\cl(\conv\{M \cup (\gamma M + v)\})$ is a {\em truncated cone} (any set that can be expressed in this form will be called a truncated cone). 

A truncated cone with $\gamma = 0$ is called a {\em pyramid} and is denoted $P(M,v)$. If $M$ is a polyhedron, then $P(M,v)$ is a {\em polyhedral pyramid}. $v$ is called the apex of $P(M,v)$ and $M$ is called the base of $P(M,v)$. The {\em height} of a pyramid $P(M,v)$ is the distance of $v$ from the affine hull of $M$.

When $M$ is a hyperplane, the truncated cone is called a {\em split}.
\end{definition}

\begin{definition}[Simplex and Generalized Simplex]\label{Def:Gen-Simplex} A simplex is the convex hull of affinely independent points. Note that a simplex is also a pyramid. In fact, any facet of the simplex can be taken as the base, and the height of the simplex can be defined with respect to this base.

A {\em generalized simplex} in $\R^n$ is given by the Minkowski sum of a simplex $\Delta$ and a linear space $X$ such that $X$ and $\aff(\Delta)$ are orthogonal to each other. Any facet of $\Delta+X$ is given by the Minkowski sum of a base of $\Delta$ and $X$. The height of the generalized simplex with respect to such a facet is defined as the height of $\Delta$ with respect to the corresponding base.
\end{definition}
We first show that $b+\Z^n$ free generalized simplices are a good class of polyhedra to approximate other $b+\Z^n$ free convex bodies within a factor that depends only on the dimension. This result is a mild strengthening of Proposition 29 in~\cite{averkov2017approximation} and the proof here is very similar to the proof of that proposition.
\begin{lemma}\label{lem:set-with-simplex} Let $n\in \N$ and $b \in \Q^n\setminus \Z^n$ such that the largest denominator in a coordinate of $b$ is $s$. Let $S = b + \Z^n$. Then for any $S$-free set $L \subseteq \R^n$, there exists an $S$-free generalized simplex $B = \Delta + X$ (see \cref{Def:Gen-Simplex}) such that $\frac{1}{s4^{n-1}\Flt(n)}L \subseteq B$. Moreover, after a unimodular transformation, $B$ has a facet parallel to $\{x \in \R^n: x_n = 0\}$, the height of $B$ with respect to this facet is at most 1, and $X = \R^{m} \times \{0\}$ for some $m < n$. 
\end{lemma}

\begin{proof} We proceed by induction on $n$. For $n=1$, all $S$-free sets are contained in a $b+\Z$ free interval, so we can take $B$ to be this interval. For $n\geq 2$, consider an arbitrary $S$-free set $L$. By \cref{thm:Flt}, $L' := \frac{1}{s4^{n-2}\Flt(n)}L$ has lattice width at most $\frac{1}{s}$. Perform a unimodular transformation such that the lattice width is determined by the unit vector $e^n$ and $b_n \in [0,1)$.  

If $b_n \neq 0$, then $b_n \in [1/s, 1-1/s]$, and therefore $L'$ is contained in the split $\{x: b_n - 1 \leq x_n \leq b_n\}$. We are done because all splits are generalized simplices and $\frac{1}{s4^{n-1}\Flt(n)}L = \frac{1}{4}L' \subseteq L' \subseteq B:= \{x: b_n - 1 \leq x_n \leq b_n\}$.

If $b_n = 0$, then $L \cap \{x : x_n = 0\}$ is an $S'$-free set in $\R^{n-1}$, where $S' = (b_1, \ldots, b_{n-1})+ \Z^{n-1}$. Moreover, by the induction hypothesis applied to $L \cap \{x : x_n = 0\}$ and $L' \cap \{x : x_n = 0\}$ it follows that there exists an $S'$-free generalized simplex $B' \subseteq \R^{n-1} \times \{0\}$ such that $L' \cap \{x : x_n = 0\} \subseteq B'$. Let $B'$ be the intersection of halfspaces $H'_1, \ldots, H'_k \subseteq \R^{n-1}$. By a separation argument between $L'$ and $\cl(\R^{n-1}\setminus H'_i) \times \{0\}$, one can find halfspaces $H_1, \ldots, H_k\subseteq \R^n$ such that $H_i \cap (\R^{n-1} \times 0) = H'_i \times \{0\}$ and $L' \subseteq H_1 \cap \ldots \cap H_k$ (this separation is possible because $0\in \intr(L')$). 

We now consider the set $P := H_1 \cap \ldots \cap H_k \cap \{x: -1/s\leq x_n \leq 1/s\}$. By construction, $P\subseteq \R^n$ is $S$-free and $L' \subseteq P$ since $L'$ has height at most $\frac{1}{s}$ and contains the origin. $P$ is also a truncated cone given by $v = \frac{2}{s}e^n$ and $M = P \cap \{x : x_n = -1/s\}$ and some factor $\gamma$ (see \cref{Def:Pyramid}), because $B'$ is a generalized simplex. Without loss of generality, one can assume $\gamma \leq 1$ (otherwise, we change $v$ to $-v$ and $M$ to $P \cap\{x : x_n = 1\}$). By applying Lemma 25 (b) in~\cite{averkov2017approximation}, one can obtain a generalized simplex $B$ as the convex hull of some point $x \in P \cap\{x : x_n = \frac{1}{s}\}$ and $M$ such that $\frac{1}{4}P \subseteq B \subseteq P$ (the hypothesis for Lemma 25 (b) in~\cite{averkov2017approximation} is satisfied because $0$ can be expressed as the mid point of two points in $P \cap\{x : x_n = \frac{1}{s}\}$ and $P \cap\{x : x_n = -\frac{1}{s}\}$ ). Since $L' \subseteq P$, we have that $\frac{1}{s4^{n-1}\Flt(n)}L = \frac{1}{4}L' \subseteq \frac{1}{4}P \subseteq B$. Since $B \subseteq P$, $B$ is $S$-free.
\end{proof}

\begin{figure}[t]
\centering
\captionsetup{justification=centering}
\captionsetup[subfigure]{justification=centering}
\begin{subfigure}[b]{0.45\textwidth}
    \includegraphics[width=\textwidth]{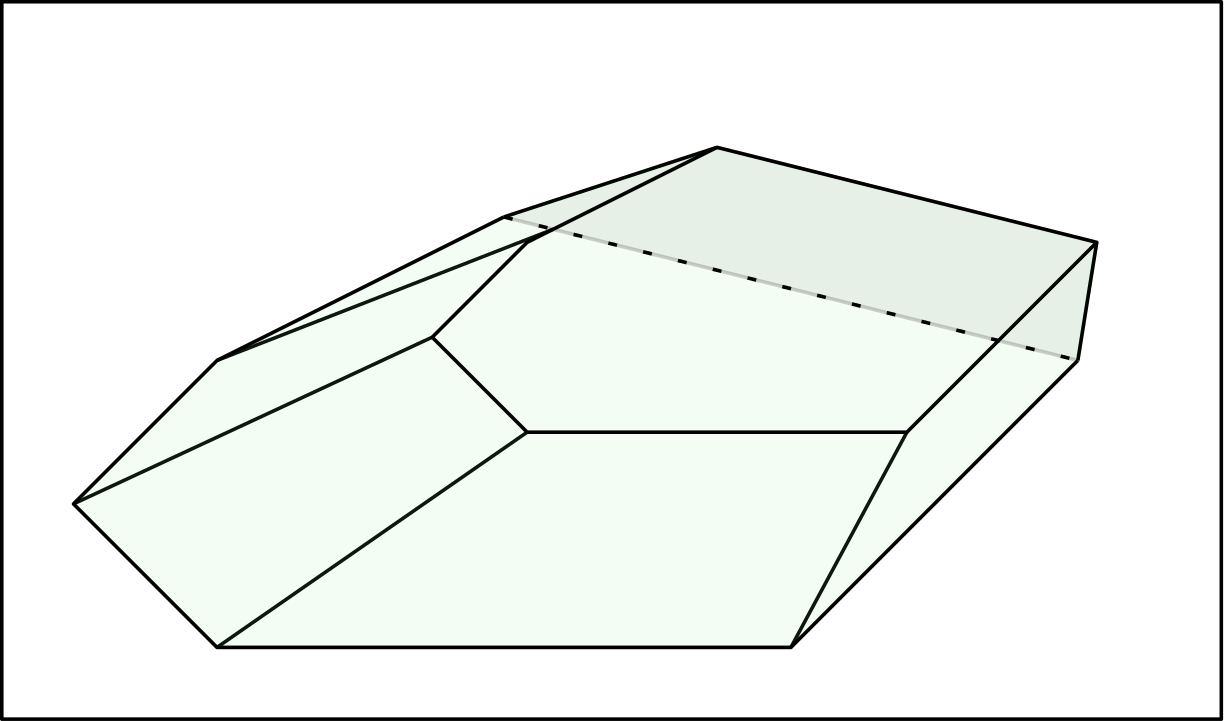}
    \caption{A $b+\Z^n$-free convex set that is to be approximated with a $b+\Z^n$-free simplex.}\label{fig:Simplex1}
\end{subfigure}\hfill
\begin{subfigure}[b]{0.45\textwidth}
    \includegraphics[width=\textwidth]{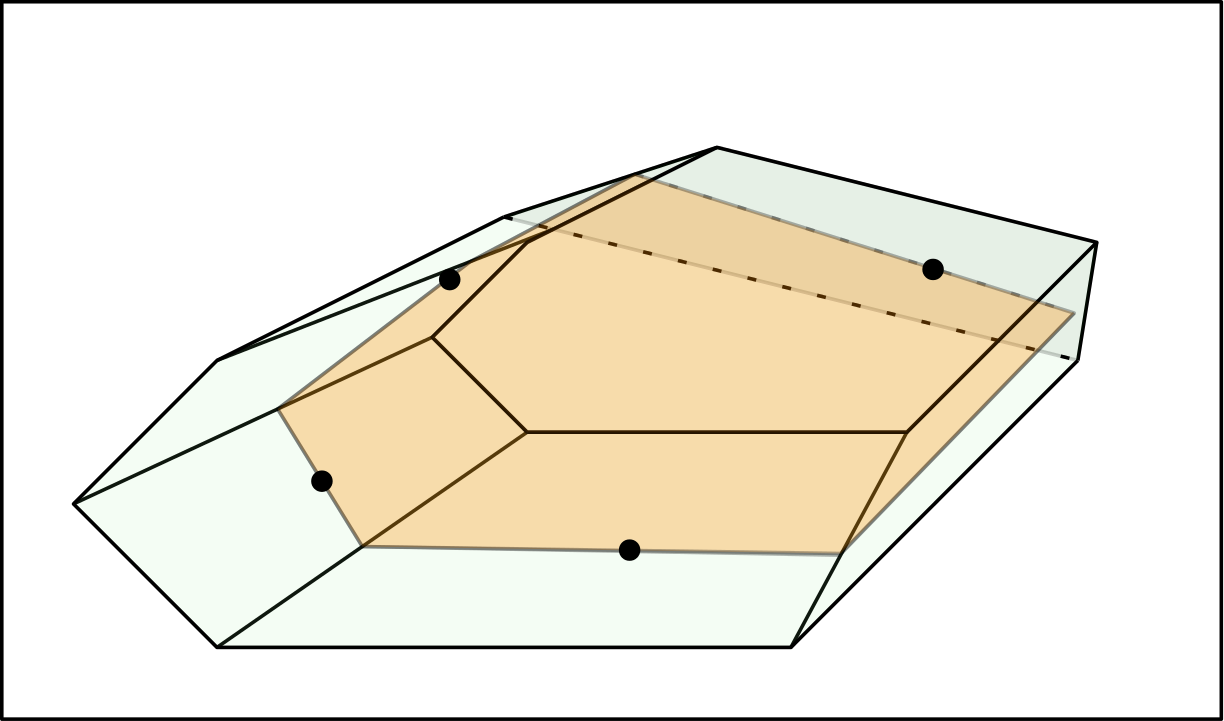}
    \caption{The integer lattice plane passing through the convex set is shown in orange.}\label{fig:Simplex2}
\end{subfigure}
\begin{subfigure}[b]{0.45\textwidth}
    \includegraphics[width=\textwidth]{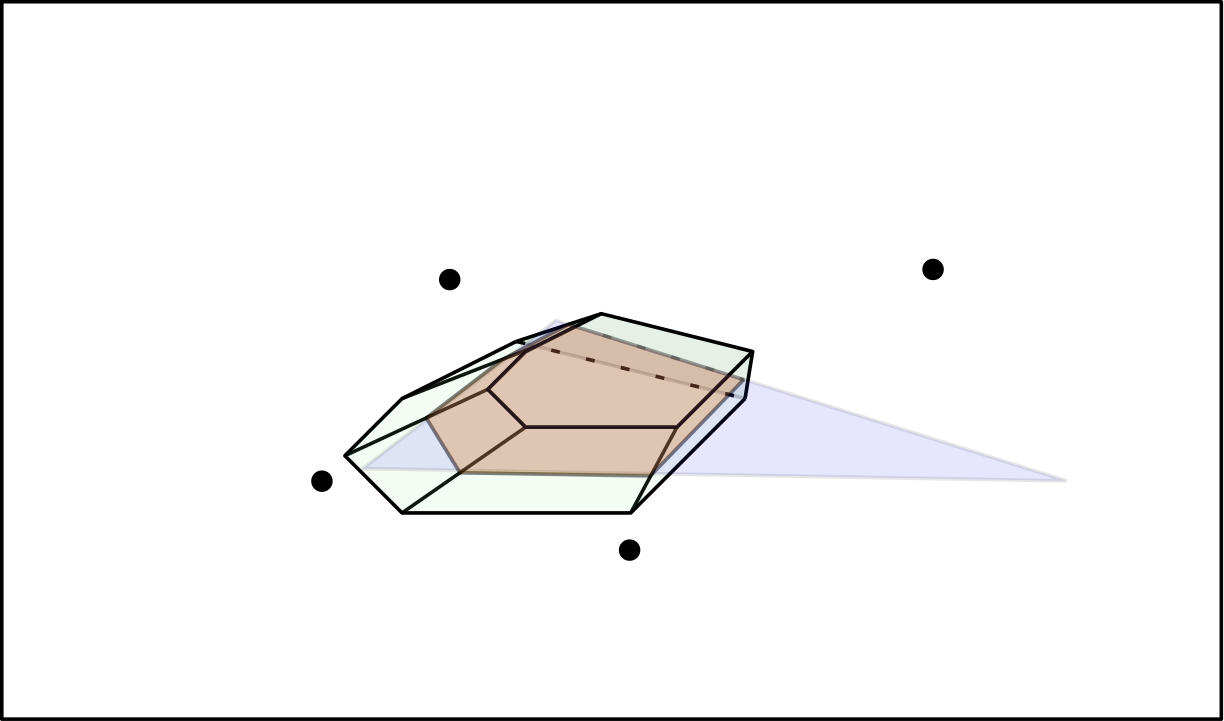}
    \caption{The set shown in orange is a \emph{lower-dimensional} $b+\Z^n$-free convex set. This can be approximated by a lower-dimensional simplex using the induction hypothesis.}\label{fig:Simplex4}
\end{subfigure}\hfill
\begin{subfigure}[b]{0.45\textwidth}
    \includegraphics[width=\textwidth]{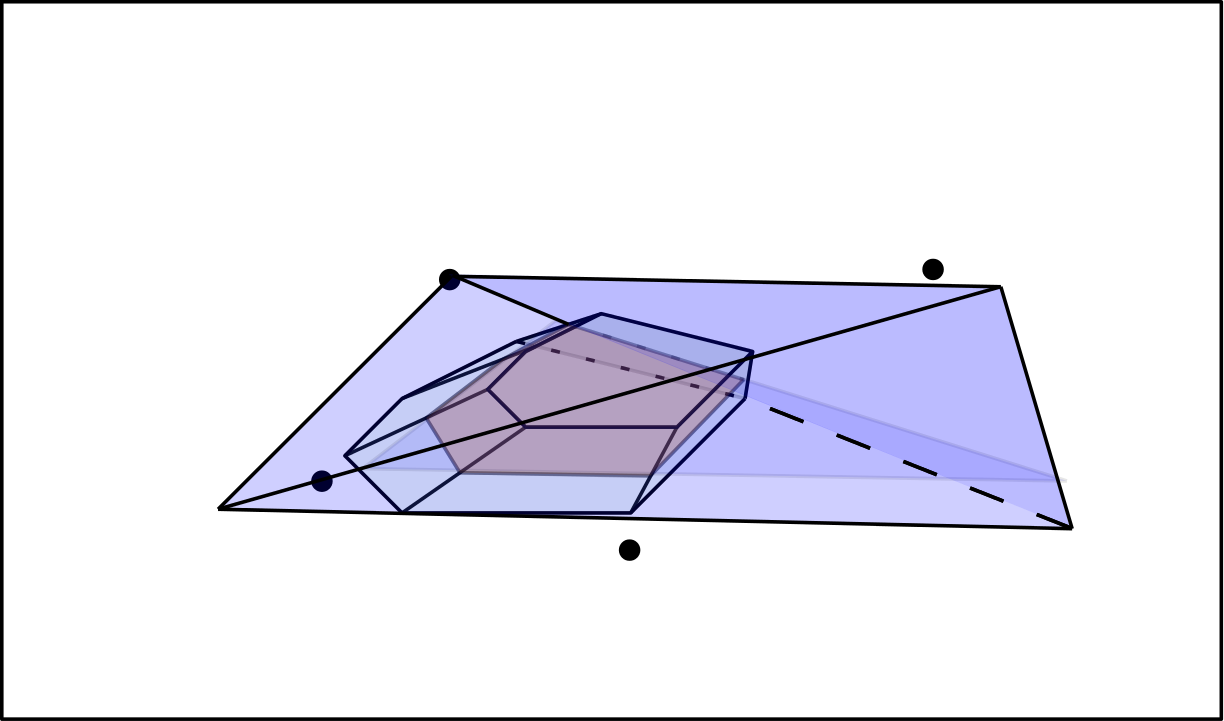}
    \caption{Hyperplanes can  be added that passes through the facets of the set in orange to get a truncated pyramid and then a simplex to approximate the given $b+\Z^n$-free set.}\label{fig:Simplex5}
\end{subfigure}
\caption[Approximating a $b+\Z^n$-free convex set with a simplex]{Intuition behind \cref{lem:set-with-simplex} to approximate a $b+\Z^n$-free convex set with a simplex.}\label{fig:SimplexXPoly}
\end{figure}

\begin{proof}[Proof of \cref{thm:gcp-main}] We proceed by induction on $n$. If $n=1$, then an $S$-free convex set is contained in an $S$-free interval, which is an $S$-free generalized cross-polyhedron, so we are done. 

For $n\geq 2$, by \cref{lem:set-with-simplex}, there exists an $S$-free generalized simplex $P = \Delta + X$ (see \cref{Def:Gen-Simplex}) such that $\frac{1}{s4^{n-1}\Flt(n)}L \subseteq P$. Moreover, after a unimodular transformation, $P$ has a facet parallel to $\{x \in \R^n: x_n = 0\}$ and the height of $P$ with respect to this facet is at most 1. Moreover, $X$ can be assumed to be $\R^{m} \times \{0\}$ for some $m < n$ {since $X$ has to be parallel to the facet defined by $x_n=0$}. Thus, by projecting on to the last $n-m$ coordinates, we may assume that $P$ is a simplex with a facet parallel to $\{x \in \R^n: x_n = 0\}$. Without loss of generality, we may assume $b_n \in [0,1)$ (by translating everything by an integer vector). We now consider two cases. 

If $b_n \neq 0$, then $b_n \in [1/s, 1-1/s]$. Moreover, $\frac{1}{s}P$ has height at most $\frac{1}{s}$, and therefore it is contained in the maximal $S$-free split $\{x: b_n - 1 \leq x_n \leq b_n\}$. We are done because all maximal $S$-free splits are generalized cross-polyhedra and $\left (\frac{1}{s4^{n-1}\Flt(n)} \right)^{n-1} L  \subseteq \frac{1}{s}P  \subseteq B:= \{x: b_n - 1 \leq x_n \leq b_n\}$.

If $b_n = 0$, then by the induction hypothesis, there exists a translated generalized cross-polyhedron $B' \subseteq \R^{n-1} \times \{0\}$ such that $\left (\frac{1}{s4^{n-2}\Flt(n-1)}\right )^{n-2}(P\cap \{x: x_n = 0\}) \subseteq B'$. Let $v$ be the vertex of $P$ with positive $v_n$ coordinate. Since the height of $P$ is at most 1, the height of $\left (\frac{1}{s4^{n-2}\Flt(n-1)}\right )^{n-2}P$ is also at most 1. Let the facet $F$ of  $\left (\frac{1}{s4^{n-2}\Flt(n-1)}\right )^{n-2}P$ parallel to $\{x \in \R^n: x_n = 0\}$ be contained in the hyperplane $\{x \in \R^n: x_n = \lambda\}$, where $-1 < \lambda < 0$ since $P$ has height at most 1 with respect to this facet. Moreover, we may assume that after a unimodular transformation, the projection of $v$ on to $\R^{n-1} \times \{0\}$ lies in $B'$, because the points from $S$ on the boundary of $B'$ form a lattice hypercube in $\R^{n-1}$ by \cref{thm:gcp-facts}(ii). Let this projected vertex be $c \in \R^{n-1}$. Let $\mu = 1- |\lambda|$ and $\gamma = \lambda$. Create the generalized cross-polyhedron $B$ from $B', c, \mu, \gamma$ in $\R^n$ as described in \cref{Def:gen-cross-poly}. By the choice of $\mu$ and $\gamma$ and the fact that $P$ has height at most 1, $v \in B$. 

We also claim that $F \subseteq (\frac{1}{\mu}(B'-c) + c) \times \{\gamma\} \subseteq B$. Indeed, observe that $$F - (c, \lambda)\subseteq \frac{1}{\mu}\bigg(\bigg(\left (\frac{1}{s4^{n-2}\Flt(n-1)}\right )^{n-2}P \cap \{x \in \R^n: x_n = 0\}\bigg) - (c,0)\bigg).$$ Since $\left (\frac{1}{s4^{n-2}\Flt(n-1)}\right )^{n-2}(P\cap \{x: x_n = 0\}) \subseteq B'$, we have $F \subseteq \left (\frac{1}{\mu}(B'-c) + c\right )   \times \{\gamma\}$.

Thus, we have that $\left (\frac{1}{s4^{n-2}\Flt(n-1)}\right )^{n-2}P \subseteq B$ since $v \in B$ and $F \subseteq B$. Combining with $\frac{1}{s4^{n-1}\Flt(n)}L \subseteq P$, we obtain that $$\left (\frac{1}{s4^{n-1}\Flt(n)} \right)^{n-1} L \subseteq \left (\left (\frac{1}{s4^{n-2}\Flt(n-1)}\right )^{n-2}\right ) \frac{1}{s4^{n-1}\Flt(n)}L \subseteq B$$
\end{proof}

\section{Algorithms for trivial lifting in generalized cross-polyhedra}
\label{sec:Algorithm_for_trivial_liftings}

The key fact that we utilize in designing an algorithm to compute the trivial liftings of generalized cross-polyhedra is the following{: generalized cross-polytopes have the so-called \emph{covering property}. We refer the readers to \cite{bcccz} for the implications that the covering property leads to and especially to \cite[Theorem 5]{bcccz} which shows that the covering property is necessary and sufficient to ensure that the trivial lifting is the unique minimal lifting.} 

{\cite[Section 4]{basu-paat-lifting} discusses the coproduct operation used to construct the generalized cross-polytopes. \cite[Theorem 4.1]{basu-paat-lifting} assures that as long as the ``initial'' sets used in the coproduct operation have the covering property, so does the final set. In our construction of generalized cross-polytopes, the corresponding initial sets are $b+\Z$ free intervals, which have the covering property.}

{Having the covering property is important for computations in the following way: it implies existence of the so-called lifting region $T$ (first defined in \cite{dw}) corresponding to the generalized cross-polyhedra such that $T+\Z^n = \R^n$. Then, one can calculate the trivial lifting at a point $x$ by calculating the gauge at $x+z$ where $z\in\Z^n$ and $x+z\in T$ (such a $z$ always exists because $T+\Z^n = \R^n$). We formalize this in the theorem below.}

\begin{theorem}\label{thm:structure-gcp-tl} Let $G\subseteq \R^m$ be any generalized cross-polytope and let $b \in \R^m\setminus \Z^m$ such that $-b \in \intr(G)$. There is a subset $T \subseteq G$ such that $T + \Z^m = \R^m$ and for any $p\in \R^m$, there exists $\tilde p \in b+T$ such that $\tilde p \in p + \Z^m$ and $\widetilde{\psi_{b+G}}(p) = \psi_{b+G}(\tilde p)$. 
\end{theorem}
Thus, for any generalized cross-polyhedron $G\subseteq \R^m$ and $p\in \R^m$, if one can find the $\tilde p$ in \cref{thm:structure-gcp-tl}, then one can compute the trivial lifting coefficient $\tilde \psi_{b+G}(p)$ by simply computing the gauge function value $\psi_{b+G}(\tilde p)$. The gauge function can be computed by simple evaluating the $2^m$ inner products in the formula $\psi_{b+G}(r) = \max_{i=1}^{2^m} a^i\cdot r$, where $a^i, i = 1, \ldots, 2^m$ are the normal vectors as per \cref{thm:gcp-facts}(i). 

Thus, the problem boils down to finding $\tilde p$ from \cref{thm:structure-gcp-tl}, for any $p\in \R^m$. Here, one uses property (ii) in \cref{thm:gcp-facts}. 
%
%
This property guarantees that given a generalized cross-polytope $G\subseteq \R^m$, there exists $\bar z \in \Z^n$ that can be explicitly computed using the $\gamma$ values used in the recursive construction, such that $T \subseteq G \subseteq \cup_{j=1}^m ((\bar z+[0,1]^m) + \ell_j )$, where $\ell_j$ is the {1-dimensional linear subspace parallel to the} $j$-th coordinate axis obtained by setting all coordinates to $0$ except coordinate $j$. Now, for any $p\in \R^m$, one can first find the (unique) translate $\hat p \in p + \Z^n$ such that $\hat p \in b+\bar z + [0,1]^m$ (this can be done since $b$ and $z$ are explicitly known), and then $\tilde p$ in \cref{thm:structure-gcp-tl} must be of the form $\hat p + Me^j$, where $M \in \Z$ and $e^j$ is one of the standard unit vectors in $\R^m$. Thus, $$\tilde\psi_{b+G}(p) = \min_{\substack{j \in \{1, \ldots, m\}, \\ M \in \Z}} \psi_{b+G}(\hat p + Me^j).$$ For a fixed $j\in \{1, \ldots, m\}$, this is a one dimensional convex minimization problem over the integers $M \in \Z$ for the piecewise linear convex function $\phi_j(\lambda) = \psi_{b+G}(\hat p + \lambda e^j) = \max_{i=1}^{2^m} a^i\cdot (\hat p + \lambda e^j)$. Such a problem can be solved by simply sorting the slopes of the piecewise linear function (which are simply $a^1_j, \ldots, a^{2^n}_j$), and finding the point $\bar \lambda$ where the slope changes sign. Then either $\phi_j(\lceil \bar\lambda\rceil)$ or $\phi_i(\lfloor \bar\lambda\rfloor)$ minimizes $\phi_j$. Taking the minimum over $j=1, \ldots, m$ gives us the trivial lifting value for $p$. 

One observes that this entire procedure takes $O(m2^m)$. While this was described only for generalized cross-polytopes, generalized cross-polyhedra of the form $G \times \R^{n-m}$ pose no additional issues: one simply projects out the $n-m$ extra dimensions.

We give a formal description of the algorithm below in \cref{alg:CoProdLift}. We assume access to procedures $\textsc{GetNormal}(G, b)$ and $\textsc{Gauge}(G,b,x)$. $\textsc{GetNormal}(G, b)$ takes as input a generalized cross-polytope $G$ and $b$ such that $-b \in \intr(G)$, and returns the list of normals $\{a^1, \ldots, a^{2^n}\}$ such that $b + G = \{x \in \R^n : a^i\cdot x \leq 1, \;\; i =1, \ldots, 2^n\}$ (property (i) in \cref{thm:gcp-facts}). $\textsc{Gauge}(G,b,r)$ takes as input a generalized cross-polytope $G$ and $b$ such that $-b \in \intr(G)$ and a vector $r$, and returns $\psi_{b+G}(r)$ (given the normals from $\textsc{GetNormal}(G, b)$, one simply computes the $2^n$ inner products $a^i\cdot r$ and returns the maximum). 

\begin{algorithm}
\begin{algorithmic}[1]
    \Require Generalized cross-polytope $G\subseteq \R^n$, $b\in \R^n\setminus\Z^n$ such that $-b\in \intr(G)$. $p \in \R^n$ where the lifting is to be evaluated.
    
    \Ensure $\widetilde{\psi_{b+G}}(p)$
    \Function{CrossPolyLift}{$G,\,\b,\,\x$}
    	\State Set $\bar z \in \R^n$ using parameters of $G$ as given in property (ii) in \cref{thm:gcp-facts}. 
        \State Compute unique $\hat p \in (p + \Z^n)\cap \b+\bar z + [0,\,1]^n$.  
        \State Let $\mathcal N = \textsc{GetNormal}(G, b)$ be the set of normals.
        \For {Each coordinate $j$ from $1$ to $n$}
        \State Find $a^{-} \in \arg\max_{a \in \mathcal N}\{a_j : a_j \leq 0\}$ where $a_j$ denotes the $j$-th coordinate of $a \in \mathcal N$). Break ties by picking the one with maximum $a\cdot \hat p$.
        \State Find $a^{+} \in \arg\min_{a \in \mathcal N}\{a_j : a_j > 0\}$ where $a_j$ denotes the $j$-th coordinate of $a \in \mathcal N$). Break ties by picking the one with maximum $a\cdot \hat p$.
        \State $\bar\lambda \gets \frac{a^+\cdot \hat p - a^-\cdot \hat p}{a^- - a^+}$. 
        \State $m_j \gets \min\{a^+\cdot \hat p + \lceil \bar\lambda \rceil a^+_j, a^-\cdot \hat p + \lfloor \bar\lambda \rfloor a^-_j\}$.
      	 \EndFor
       \State \Return $\min\{1,m_1, \ldots, m_j\}$.
    \EndFunction
\end{algorithmic}
\caption{Trivial lifting of a generalized cross-polytope}
\label{alg:CoProdLift}
\end{algorithm}

\section{Computational Experiments and Results} \label{sec:CompExpr}
{In this section we give results from a set of computational experiments comparing the cuts described in this paper against Gomory's Mixed Integer (GMI) cuts, {and also CPLEX computations at the root node}. We perform four types of computational tests:}
\begin{enumerate}
	\item {Testing on random dense instances of pure-integer and mixed-integer programs.}
	\item {Testing stable-set problem instances and vertex-cover problem instances in random graphs.}
	\item {Testing on MIPLIB3.0 problem instances.}
	\item {Testing an approximation to the closure of the cuts from all generalized cross-polytopes, on the random dense instances of mixed-integer programs.} 
	\end{enumerate}
{In the following subsections, we describe the terms used above, the exact testing procedure adopted and our results in these problems.  We observe that, despite the strong theoretical results, the performance of the cuts derived from the generalized cross-polytopes {in our particular computational set-ups is generally poor. As mentioned in the Introduction, we suspect that this is because our naive sampling of the cuts is not good enough and the cut selection problem for this family we propose is still a non trivial problem.} }

\subsection{Test on random dense instances}
\label{subsec:Dense}

First we describe the test we performed on random dense instances of pure and mixed-integer programs. We describe our problem generation procedure, cut generating procedure, comparison procedure in the following paragraphs. The testing procedure is also summarized in \cref{alg:CompTest}.


\subsubsection{Data generation}
e write all our test problems in the canonical form 
\begin{align}
\min_{x\in\R^d}&\left\{c^Tx : Ax = b; x\geq 0; i\in\mathcal{I} \implies x_i\in\Z \right \}
\end{align}
where $A\in\R^{k\times d}, b\in\R^k, c\in\R^d$ and $\mathcal{I}\subseteq \{1,2,\ldots,n\}$.\\

We generated {roughly 12,000} problems in the following fashion. 
\begin{itemize}
    \item Each problem can be pure integer or mixed integer. For mixed-integer problem, we decide if each variable is discrete or continuous randomly with equal probability.
	\item{Each problem can have the data for $A,\,b$ and $c$ as matrices with either integer data or rational data. {Each entry is uniformly distributed between -10 and 10. In the former case, only integers are considered and in the latter case, rational numbers represented upto 8 decimal places are considered. Thus the matrix $A$ is a dense matrix.}}
    \item{The size of each problem varies from $(k,\,d) \in \left \lbrace (10i,\,25i): i\in\left \lbrace 1,\,2,\,\ldots,\,10\right \rbrace\right \rbrace$.}
    \item{There are roughly $300$ realizations of each type of problem.}
\end{itemize}
This leads to $2\times 2\times 10\times 300 \,(\text{roughly}) \approx 12,000$ problems in all. {The entire data set can be found at this hyperlink: \href{http://www.ams.jhu.edu/~abasu9/Data_Sets/}{http://www.ams.jhu.edu/$\sim$abasu9/Data\_Sets/}.} This number is not precise as some random problems where infeasibility or unboundedness were discovered in the LP relaxation were ignored. Below we present the results for these approximately 12,000 problems as a whole and also the performance of our methods in various subsets of these instances.

\subsubsection{Cut generation}\label{sec:cut-generation}
We consider three types of cuts in these computational tests - Gomory's mixed-integer (GMI) cuts, X-cuts and GX-cuts. GMI cuts are single row cuts obtained from standard splits~\cite[Eqn 5.31]{conforti2014integer}. 
GX-cuts are cuts obtained from certain structured generalized cross-polytopes defined in \cref{Def:gen-cross-poly}. X-cuts are obtained from a special case of generalized cross-polytopes, where the center $(c,\,\gamma)$ coincides with the origin. It should be noted that the GMIs are indeed a special case of X-cuts, because they can be viewed as cuts obtained from {$b+\Z^n$ free} intervals or one-dimensional generalized cross-polytopes whose center coincide with the origin. In this section, we call such cross-polytopes as \emph{regular cross-polytopes}. This motivates the set inclusions shown in \cref{fig:CutTypeVenn}. The motivation behind classifying a special family of cross-polytopes with centers coinciding with the origin is the algorithmic efficiency they provide. Because of the special structure in these polytopes, the gauges and hence the cuts can be computed much faster than what we can do for an arbitrary generalized cross-polytope {(comparing with the algorithms in \cref{sec:Algorithm_for_trivial_liftings})}. In particular, the gauge and the trivial lifting can both be computed in $O(n)$ time, as opposed to $O(2^n)$ and $O(n2^n)$ respectively for the general case (see \cref{sec:Algorithm_for_trivial_liftings}), where $n$ is the dimension of the generalized cross-polytopes or equivalently, the number of rows of the simplex tableaux used to generate the cut.

The family of generalized cross-polytopes that we consider can be parameterized by a vector $\mu\in (0,1)^n$ and another vector in $f\in\mathbb{R}^n$. This vector consists of the values $\mu_i$ used in each stage of construction of the cross-polytope, after appropriate normalization (see \cref{Def:gen-cross-poly}). This actually forces $\sum_{i=1}^{n}\mu_i = 1$. The vector $f$ corresponds to the center of the generalized cross-polytope; the coordinates of $f$ give the coordinates of $c$ and $\gamma$ in the iterated construction of \cref{Def:gen-cross-poly}. Both the parameters $\mu$ and $f$ show up in \cref{alg:CompTest}. The regular cross-polytopes are obtained by setting $f=\mathbf{0}$ in the above construction; thus, they are parameterized by only the vector $\mu\in (0,1)^n$. As long as $\sum_{i=1}^{n}\mu_i = 1$, there exists a one-to-one map between such vectors and the set of regular cross-polytopes in $\mathbb{R}^n$. 
\begin{figure}[t]
    \centering
    \includegraphics[width=0.5\textwidth]{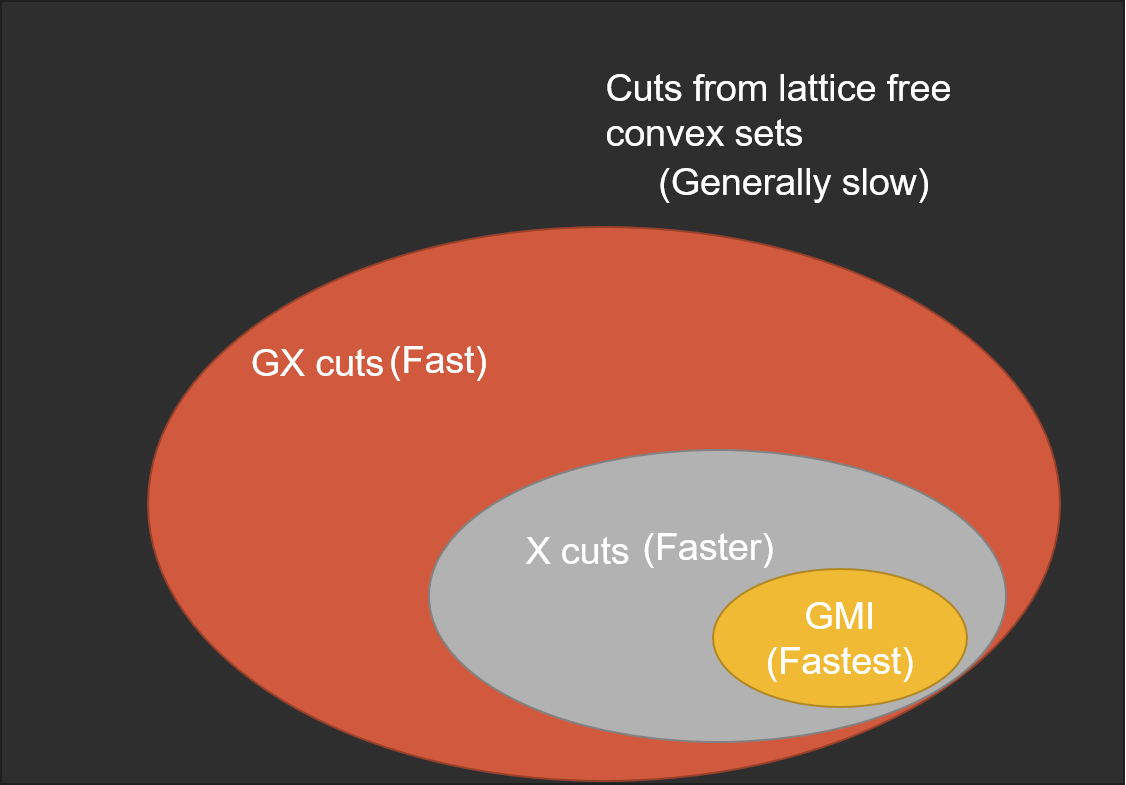}
    \caption{Venn diagram showing inclusions of various types of cuts and algorithmic efficiencies to generate them.}\label{fig:CutTypeVenn}
\end{figure}

\begin{algorithm}[h]
    \begin{algorithmic}[1]
        \Require A mixed-integer problem (MIP) in standard form. Number $N \geq 2$ of rows to use to generate multi-row cuts; Number $k\geq 1$ of multi-row cuts; Number $\ell\geq 1$ of rounds of multi-row cuts to be used; Number of $1 \leq q \leq N$ non-integer basics to be picked for GX-cuts.
        \State   \lp$\gets$ Objective of LP relaxation of MIP.
        \State{In the final simplex tableaux, apply GMI cuts on all rows whose corresponding basic variables are constrained to be integer in the original problem, but did not turn out to be integers. }
        \State \gmi $\gets$ Objective of LP relaxation of MIP and GMI cuts.
        \For {$i$ from $1$ to $\ell$}
            \For {$j$ from $1$ to $k$}
                \State Generate $\mu \in [0,\,1]^N$ such that $\sum_{\xi = 1}^N\mu_\xi = 1$. Also randomly select $N$ rows where integrality constraints are violated for corresponding basic variables.
                \State Generate an X-cut from the generated $\mu$ and the chosen set of rows.
                \State Generate $f \in [0,\,1]^N$ randomly. 
                \State Randomly select rows such that $q$ of them correspond to rows that violate the integrality contraints and $N-q$ of them don't.
                \State Generate a GX-cut from the generated $\mu,\, f$ and the set of rows.
            \EndFor
            \State $\X_i\gets$ Objective of LP relaxation of MIP and all the X-cuts generated above.
            \State $\XG_i\gets$ Objective of LP relaxation of MIP with all the X-cuts as well as the GMI cuts.
            \State $\GX_i\gets$ Objective of LP relaxation of MIP and all the GX-cuts generated above.
            \State $\GXG_i\gets$ Objective of LP relaxation of MIP with all the GX-cuts as well as the GMI cuts.
        \EndFor
        \State $\X\gets \max_{i=1}^\ell \X_i$; $\XG\gets \max_{i=1}^\ell \XG_i$; $\GX\gets \max_{i=1}^\ell \GX_i$; $\GXG\gets \max_{i=1}^\ell \GXG_i$. 
        \State $\Best \gets \max \left \lbrace \X, \XG, \GX, \GXG \right \rbrace$
        \State \Return $\lp,\, \gmi,\, \X,\,\XG,\,\GX,\,\GXG, \Best$
    \end{algorithmic}
\caption{Computational testing procedure}
\label{alg:CompTest}
\end{algorithm}

{We also note that \emph{any} cut generated from a generalized cross-polytope, or for that matter, any valid pair (see \cref{Def:ValidPair}) cuts off the fractional solution. This is because, the fractional solution obtained corresponds to $s=0$ and $y=0$ in the space $(s,y)$ using the notation in \eqref{eq:cut}. So no matter what the values of $\psi(r_i)$ and $\pi(p_i)$ are, the LHS of the inequality in \eqref{eq:cut} is 0 for the fractional point. Thus the {current fractional LP solution} is \emph{always} separated from the convex hull.}

\subsubsection{Comparison procedure}

In each of the problems, the benchmark for comparison was an aggressive addition of GMI cuts. The procedure used for comparison is mentioned in \cref{alg:CompTest}. We would like to emphasize that X-cuts and GX-cuts are an infinite family of cuts unlike the GMI cuts. However, we add only finitely many cuts from this infinite family. 

In all the computational tests in this paper, these cuts are randomly generated without looking into any systematic selection of rows or $\mu$. However to improve the performance from a completely random selection, we generate $\ell$ batches of $k$ cuts and only keep the best set of $k$ cuts. We lay out our testing procedure in detail in \cref{alg:CompTest}.

For the set of 12,000 problems, X-cuts and GX-cuts were generated with $N = 2,\, 5, \text{ and }10$ rows. For GX-cuts, the number $q$ of rows to be picked whose corresponding basic variables violate integrality constraints, was chosen to be 1. This was found to be an ideal choice under some basic computational tests with small sample size, where cuts with different values of $q$ were compared. {Also, a qualitative motivation behind choosing $q=1$} is as follows: GMI cuts use information only from those rows where integrality constraints { on the corresponding basic variables} are violated. To beat GMI, { it is conceivably more} useful to use information not already available for GMI cuts, and hence to look at rows where the integrality constraint {on the corresponding basic variable} is not violated.

\subsubsection{Results}\label{sec:random-results} 
A typical measure used to compute the performance of cuts is \emph{gap closed} which is given by $\frac{\mathsf{cut}-\lp}{\mathsf{IP}-\lp}$. However the IP optimal value $\mathsf{IP}$ could be expensive to compute on our instances. So, as a first test, we use a different metric, which compares the performance of the best cut we have, against that of GMI cuts. Thus we define 
 \begin{align}
        \beta \quad&=\quad \frac{\Best - \gmi}{\gmi - \lp},
            \end{align} 
            which tries to measure the {\em improvement} over GMI cuts using the new cuts.


\begin{table}[h]
\centering
 \caption{Results}
\begin{tabularx}{\textwidth}{|p{4cm}|X|X|X|X|}
    \hline
    \textbf{Filter} & \textbf{Number of problems} & \textbf{Cases where $\gmi < \Best$} & \textbf{Average of $\beta$} & \textbf{Average of $\beta$ when GMI is beaten}\\
    \hline
    None (All problems) & 13604  & 6538 (48.06\%) & 2.00\% & 4.15\% \\
    \hline
    Rational Data & 6600 &  3213 (48.68\%) & 2.11\% & 4.23\%\\
    \hline
    Integer Data & 7004 &  3325 (47.47\%) & 1.90\% & 3.80\%\\
    \hline
    \parbox{4cm}{Pure Integer problems} & 6802 &  2189 (32.18\%) & 0.69\% &2.146\% \\
    \hline
    \parbox{4cm}{Mixed Integer problems} & 6802 &  4376 (64.33\%) & 3.32\% & 5.159\%\\
    \hline
    \parbox{4cm}{Rational Data \\ Pure Integer problems} & 3300 &  1078 (32.67\%) & 0.75\% &2.306\% \\
    \hline
    \parbox{4cm}{Rational Data\\ Mixed Integer problems} & 3300 &  2135 (64.70\%) & 3.48\% & 5.376\%\\
    \hline
    \parbox{4cm}{Integer Data \\ Pure Integer problems} & 3502 &  1111 (31.52\%) & 0.63\% &1.996\% \\
    \hline
    \parbox{4cm}{Integer Data\\ Mixed Integer problems} & 3502 &  2241 (63.42\%) & 3.17\% & 4.95\%\\
    \hline
\end{tabularx}
\label{tab:results}
\end{table}

The testing procedure mentioned in \cref{alg:CompTest} was run with the values of $k=\ell=5$. The results hence obtained are mentioned in \cref{tab:results}. Besides this table, we present some interesting observations from our computational testing.

\begin{enumerate}
    \item{In mixed-integer problems, we have $\beta \geq 10\%$ in 648 cases (which is 9.53\% of the set of mixed-integer problems). In pure-integer problems we have $\beta\geq5\%$ in 320 cases (which is 4.7\% of the set of pure-integer problems).} A conclusion from this could be that the family of cuts we are suggesting in this paper works best when we have a good mix of integer and continuous variables. We would like to remind the reader that in the mixed-integer examples we considered, roughly half the variables were continuous, due to a random choice between presence or absence of integrality constraint for each variable.

    \item{We also did some comparisons between $N=2, 5, 10$ row cuts. In particular, let us define $\beta_2,\, \beta _5$ and $\beta_{10}$ as the values of $\beta$ with $N=2, 5, 10$ respectively. Among the 13,604 cases, only in 265 cases we found $\beta_5 > \beta_2$ or $\beta_{10} > \beta_2$ (the inequalities are considered strictly here). In 264 of these cases, $\max\{\beta_5, \beta_{10}\} > \gmi$ (the inequality is strict here).
 In these 265 cases, 62 were pure-integer problems and GMI was beaten in all 62 problems. The other 203 cases were mixed integer problems. GMI was beaten in 202 of these problems.}

    We conclude that when cuts derived from higher dimensional cross-polytopes dominate cut obtained from lower dimensional cross-polytopes, then the cuts from the higher dimensional cross-polytopes dominate GMI cuts as well. In other words, if we find a good cut from a high dimensional cross-polytope, then we have a very useful cut in the sense that it adds significant value over GMI cuts. 
    
    \item{Another test was done with increasing the number $k$ which corresponds to the number of GX cuts added, from a constant 10 to half the number of GMI cuts in the problem (recall that for the results reported in \cref{tab:results}, $k = 5$). Integer data was used in this, and this test was performed in a smaller randomly generated sample of size 810. In pure integer cases, we beat GMI in about 25\% cases and in mixed-integer problems, we beat GMI in 61\% cases. The value of $\beta$ is comparable to \cref{tab:results} in both cases. But the lack of significant improvement suggests the following. The performance of cross-polytope based cuts is {determined more by the problem instance characteristics, rather than the choice of cuts}. If these cuts work well for a problem, then it should be reasonably easy to find a good cut. }
	\item{ Further there were 4 problems, all mixed-integer, with $\beta > 100\%$ suggesting {potential} that there could be a set of problems on whom a very good choice of rows and $\mu$ could give a non-trivial improvement {over the GMI cuts}.}
    \item{As far as the time taken to run these instances goes, for the number of rows considered in this test, most of the time is typically spent in solving the LP relaxation after addition of cuts, accessing the simplex tableaux {to generate the cut} etc., rather than actually computing the cut.}
\end{enumerate}

\subsection{Performance in random graph instances}
{Inspired by the notion that most of the integer programming problems of interest are sparse and have an underlying structure in them, we tested the cuts from the family of generalized cross-polyhedra on two graph problems namely, the stable set problem and the vertex cover problem. Both these problems are NP-complete by themselves and can be posed as an IP. Let $G = (V, E)$ be a graph. \Cref{eq:STABLESET} is the stable set problem and \cref{eq:VERTEXCOVER} is the vertex cover problem.} 
\begin{subequations}
\begin{align}
	\max_{x_v} \quad&:\quad \sum_{v\in V}x_v & \text{subject to }\\
	x_u + x_v \quad&\leq\quad 1 &\forall\,e=uv\in E\\
	x_v \quad&\in\quad \{0,1\} &\forall \,v\in V
\end{align}
	\label{eq:STABLESET}
\end{subequations}
\begin{subequations}
	\begin{align}
		\min_{x_v} \quad&:\quad \sum_{v\in V}x_v &\text{subject to}\\
		x_u + x_v \quad&\geq\quad 1 &\forall \,e=uv\in E\\
		x_v \quad&\in\quad \{0,1\} &\forall \,v\in V
	\end{align}
	\label{eq:VERTEXCOVER}
\end{subequations}
{We generated the graphs as follows. We fixed the number of vertices $|V|$ and generated an edge $e$ with a probability $p$. We generate 100 such instances for each value of $|V|$ and $p$. We varied $|V|$ from $5, 6, \ldots, 15$ and $p$ from $0.1$ to $0.9$ in increments of $0.1$. Both the stable set problem in \cref{eq:STABLESET} and the vertex cover problem in \cref{eq:VERTEXCOVER} problem were solved for these graphs. }

{We adopted a testing procedure analogous to the procedure mentioned in \cref{subsec:Dense}. However, in this setting, we never observed \emph{any} improvement whatsoever beyond the gain obtained using GMI cuts. }

\subsection{Performance in MIPLIB 3.0}
Our testing with the new cuts discussed in this paper had meagre to no improvement in most of MIPLIB problems. Apart from the type of test mentioned in \cref{alg:CompTest} above, we performed the following test motivated by~\cite{espinoza2010computing}. We ran the MIPLIB problem on CPLEX 12.7.1, stopping after all root node calculations before any branching begins (CPLEX typically adds several rounds of cuts at the root node itself). We keep count of number of cuts added by CPLEX. Now we allow up to 10 times the number of cuts added by CPLEX, iteratively solving the LP relaxation after the addition of each cut. In each round, the cut that gives the best $\beta$ among twenty five randomly generated cut is added. We count the number of cuts we had to add and hence the number of rounds of LP we solve, to obtain an objective value as good as CPLEX. However, in almost all cases adding even ten times as many cuts as CPLEX did, did not give us the objective value improvement given by CPLEX. 

Tests along the line of \cref{alg:CompTest} were also not promising. The only set of exceptions is the {\tt enlight} set of problems in MIPLIB 3.0. These are problems coming from the Enlight combinatorial game. The X-cuts did not show any improvement over GMI cuts. The performance of the GX-cuts are shown below in \cref{tab:Enlight}. It can be seen from \cref{tab:Enlight} that the performance of GX cuts increases with the number of rows used.

{We note that we want to test the efficacy of our general purpose cutting planes, and therefore avoid using any knowledge of the structure of the MIPLIB problems in our cut generation procedure. While there could certainly be a way to use problem structure in deploying these cuts better in practice, we consider this more sophisticated approach to be outside the scope of this current manuscript.}

\begin{table}[h]
\centering
\caption{Performance on Enlight problems. {The numbers reported are the optimal values of the LP after the corresponding cuts have been added (they are minimization problems).}}
\begin{tabularx}{\textwidth}{|X|p{1cm}|p{1cm}|X|X|X|X|}
    \hline
    {\bf Problem} & {\bf \lp} & {\bf \gmi} & {\bf 2 row \GX} & {\bf 5 row \GX} & {\bf 10 row \GX} &{\bf \textsf{IP}}\\
    \hline
    enlight9 & 0 & 1 & 1.1902 & 1.4501 & 1.9810 & INF\\
    \hline
    enlight13 & 0 & 1 & 1.1815 & 1.5410 & 1.9704 & 71 \\
    \hline
    enlight14 & 0 & 1 & 1.1877 & 1.5051 & 1.9195 & INF\\
    \hline
    enlight15 & 0 & 1 & 1.2001 & 1.4712 & 1.8991 & 69 \\
    \hline
    enlight16 & 0 & 1 & 1.1931 & 1.4934 & 1.8766 & INF\\
\hline
\end{tabularx}
\label{tab:Enlight}
\end{table}

\subsection{Approximating the exact closure of generalized cross polyhedra cutting planes}\label{sec:approximate-closure}
{Using the $\beta$ metric defined above, we see that the most significant improvement is on dense random instances. Thus, we tried to do a little more intensive testing on random dense instances by apronximating the exact closure of our family as best as we could. In other words, this is an attempt to optimize the linear function over the closure of the family of cuts  obtained from generalized cross-polyhedra. Because of the nonlinear relation between the cut coefficients and the parameter $\mu$ used in defining the cross-polytope, implementing an exact separation oracle to solve this problem requires us to solve a nonlinear optimization problem. Moreover, the bigger hurdle seems to be the lack of any easy way to decide which rows should be selected to generate the separating cut from the family. This makes the separation problem for the exact closure a large mixed-integer nonlinear optimization problem which we did not see an efficient way to solve. To simulate the effect of the exact closure, we instead add a large number ($\sim$1000) of random cuts from this family and compute the gap closed.}

{Since we are adding a lot of cuts compared to GMI, it makes more sense to consider the overall gap closed with respect to the optimal IP solution, as opposed to using the $\beta$ metric. For large random dense instances, solving the IP to optimality is usually very difficult. So we decided to focus on set of about 200 random instances with 40 constraints and 100 variables.}

{In these 200 problems, the gap closed given by $\frac{\Best-\lp}{\ip-\lp}$ is of the order of {5.51\%}. In comparison, GMI cuts already close {5.04\%} of the gap. While this improvement is not very large, it is non trivial, in our opinion. It seems to complements the 10\% improvement we saw in 10\% of the cases when evaluating using the $\beta$ metric (see point 1. in the discussion in Section~\ref{sec:random-results}). With the approximate closure this 10\% improvement (going from $\sim$5\% to $\sim$5.5\%) is now seen to be an average phenomenon as opposed to only in 10\% of the cases. Of course, one has to keep in mind that the approximate closure of our family uses a lot more cuts than the GMI closure; on the other hand, we are looking at gap closed as opposed to the $\beta$ metric now, so these numbers still tell us something about our family. 
}


\section{Limitation of the trivial lifting: Proof of Theorem~\ref{thm:bad-approx}}\label{sec:bad-approx}

In this section, we show that for a general $b+\Z^n$ free set, the trivial lifting can be arbitrarily bad compared to a minimal lifting. We first show that for $n=2$, there exist $b\in \R^2\setminus\Z^2$ such that one can construct maximal $(b+\Z^2)$-free triangles with the desired property showing that the trivial lifting of its gauge function can be arbitrarily worse than a minimal lifting.

\paragraph{Example in 2 dimensions:} Consider the sequence of Type 3 Maximal $b+\Z^n$ free triangles with $b = (-0.5, -0.5)$ given by the equations
\begin{subequations}
    \begin{align}
        20x -y+ 10.5 \quad&=\quad 0 \\
        \alpha_ix + y + \frac{1-\alpha_i}{2} \quad&=\quad 0 \\
        -\beta_ix + y + \frac{1+\beta_i}{2} \quad&=\quad 0 
    \end{align}
\end{subequations}
with $\alpha_i = 1+\frac{1}{i}$ and $\beta_i = \frac{1}{i}$. Let us call the sequence of triangles as $T_i$. The triangle $T_1$ is shown in Fig. \ref{fig:BadTrivialLift}. 

For all $i$, the point $p = (0.25,0)$ is located outside the region $T_i + \Z^n$. So clearly for all $i$, the trivial lifting evaluated at $p$ is at least 1. However, let us consider the minimum possible value any lifting could take at $p$. This is given by (see~\cite[Section 7]{dw}, \cite{fixing-region}):
\begin{align}
    \pi_{\min}(p) \quad&=\quad \sup_{\substack{z\in \mathbb{Z}^n\\ w\in \mathbb{R}^n\\ w+Np \in b+\mathbb{Z}^n}}\frac{1 - \psi_{T_i}(w)}{N}\\
    \quad&=\quad \sup_{\substack{N\in \mathbb{N}\\z\in \mathbb{Z}^n}} \frac{1-\psi_{T_i}(b-Np+z)}{N}\\
    \quad&=\quad \sup_{{N\in \mathbb{N}}} \frac{1-\inf_{z\in \mathbb{Z}^n}\psi_{T_i}(b-Np+z)}{N}\\
    \quad&=\quad \sup_{{N\in \mathbb{N}}} \frac{1-\widetilde{\psi_{T_i}}(b-Np)}{N}
\end{align}
In the current example, $b = (-0.5, -0.5)$ and $p = (0.5, 0)$. Hence points of the form $b - Np$ correspond to a horizontal one-dimensional lattice. i.e., points of the form $(-(N+1)/2,\, -0.5)$. Since all of these points are arbitrarily close to the side of $T_i + z$ for some $z\in \Z^2$ (as $i \to \infty$), $\widetilde{\psi_{T_i}}(b-Np) \geq 1 - \epsilon_i$ where $\epsilon_i \rightarrow 0$. This implies that the minimal lifting of the point could become arbitrarily close to zero, and the approximation $\frac{\widetilde{\psi}(p)}{\pi_{\min}(p)}$ could be arbitrarily poor.

The proof for general $n\geq 2$ can be completed in two ways. One is a somewhat trivial way, by considering cylinders over the triangles considered above. A more involved construction considers the so-called {\em co-product} construction defined in~\cite{averkov2015lifting,basu-paat-lifting}, where one starts with the triangles defined above and iteratively takes a co-product with intervals to get maximal $b+\Z^n$ free sets in higher dimensions. It is not very hard to verify that the new sets continue to have minimal liftings which are arbitrarily better than the trivial lifting, because they contain a lower dimension copy of the triangle defined above. We do not provide more details, because this will involve definitions of the coproduct construction and other calculations which do not provide any additional insight, in our opinion.

\begin{figure}[t]
    \centering
    \includegraphics{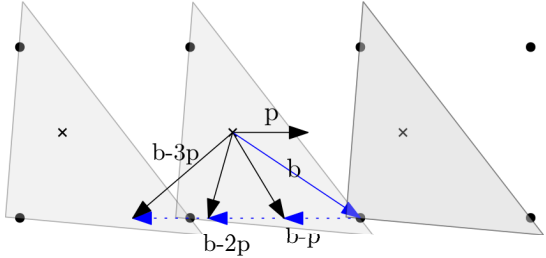}
    \caption{Example where trivial lifting can be very poor}
    \label{fig:BadTrivialLift}
\end{figure}

\bibliographystyle{siamplain}
\bibliography{library,full-bib}
\end{document}